\documentclass[reqno, 10pt]{amsart}
\usepackage{amssymb}
\usepackage{amsmath}
\usepackage[mathscr]{euscript}
\usepackage{hyperref}
\usepackage{graphicx}
\usepackage[normalem]{ulem}
\usepackage{enumitem}

\usepackage{amssymb}
\usepackage{hyperref}
\RequirePackage[dvipsnames]{xcolor} 
\definecolor{halfgray}{gray}{0.55} 
\definecolor{webgreen}{rgb}{0,0.5,0}
\definecolor{webbrown}{rgb}{.6,0,0} \hypersetup{%
  colorlinks=true, linktocpage=true, pdfstartpage=3,
  pdfstartview=FitV,%
  breaklinks=true, pdfpagemode=UseNone, pageanchor=true,
  pdfpagemode=UseOutlines,%
  plainpages=false, bookmarksnumbered, bookmarksopen=true,
  bookmarksopenlevel=1,%
  hypertexnames=true,
  pdfhighlight=/O,
  urlcolor=webbrown, linkcolor=RoyalBlue,
  citecolor=webgreen, 
  pdftitle={},%
  pdfauthor={},%
  pdfsubject={2000 MAthematical Subject Classification: Primary:},%
  pdfkeywords={},%
  pdfcreator={pdfLaTeX},%
  pdfproducer={LaTeX with hyperref}%
}

\usepackage{enumitem}

\theoremstyle{plain}
\newtheorem{theorem}{Theorem}[section]
\newtheorem{lemma}[theorem]{Lemma}
\newtheorem{corollary}[theorem]{Corollary}
\newtheorem{proposition}[theorem]{Proposition}

\theoremstyle{definition}

\newtheorem{remark}[theorem]{Remark}
\newtheorem{example}[theorem]{Example}

\DeclareMathOperator{\Id}{Id}

\def\Z{\mathbb{Z}}
\def\R{\mathbb{R}}
\def\N{\mathbb{N}}

\begin{document}

\title{Smooth Linearization of Nonautonomous Coupled Systems}

\begin{abstract} 
In a joint work with Palmer \cite{BDP} we have formulated sufficient conditions under which there exist continuous and  invertible transformations of the form $H_n(x,y)$ taking solutions of  a coupled system
\begin{equation*}
x_{n+1} =A_nx_n+f_n(x_n, y_n), \quad y_{n+1}=g_n( y_n),
\end{equation*}
onto the solutions of the associated partially linearized uncoupled system
\begin{equation*}
x_{n+1} =A_nx_n, \quad y_{n+1}=g_n( y_n).
\end{equation*}
In the present work we go one step further and provide conditions under which  $H_n$ and $H_n^{-1}$ are smooth in one of the variables $x$ and $y$.
 We emphasise that our conditions are of a general form and do not involve any kind of dichotomy, nonresonance or spectral gap assumptions for the linear part which are present on most of the related works.
\end{abstract}

\author{Lucas Backes}
\address{\noindent Departamento de Matem\'atica, Universidade Federal do Rio Grande do Sul, Av. Bento Gon\c{c}alves 9500, CEP 91509-900, Porto Alegre, RS, Brazil.}
\email{lucas.backes@ufrgs.br} 

\author{Davor Dragi\v cevi\'c}
\address{Faculty of Mathematics, University of Rijeka, Croatia}
\email{ddragicevic@math.uniri.hr}

\date{\today}


\maketitle

\section{Introduction}
\subsection{Autonomous linearization}
The classical Grobman-Hartman theorem~\cite{G1,G2,H1,H2} is one of the most important results in the qualitative theory of dynamical systems and differential equations. It  asserts that a $C^1$-diffeomorphism of $\R^n$  is topologically equivalent (or conjugated) to the associated linear part on a neighborhood of a hyperbolic fixed point. This result was later generalized  to
Banach spaces independently by  Palis~\cite{Palis} and Pugh~\cite{Pugh}. It is well-known that in general,  the conjugacy is only locally H\"{o}lder continuous~\cite{SX} and that it can fail to exhibit higher regularity even for $C^\infty$ dynamics.  On the other hand, topological conjugacy is often not sufficient  since for example it can fail to  distinguish a node from a focus as pointed out by
van Strien~\cite{Stri-JDE90}. 

Therefore, it was important to study the problem of formulating sufficient conditions that would ensure that the conjugacy in the Grobman-Hartman theorem exhibits higher regularity properties. The first results in this direction are due to Sternberg~\cite{Stern,Stern2} who
proved that $C^k$ ($k\ge1$) diffeomorphisms can be $C^r$ linearized
near  hyperbolic fixed points, where the integer $r>k$ depends on $k$ and so-called nonresonant conditions. Thus, in order to obtain $C^r$-linearization his results required  that the dynamics exhibits higher regularity. Later Belitskii~\cite{Bel73,Bel78} gave conditions for $C^k$ linearization of $C^{k,1}$ ($k\ge 1$)
diffeomorphisms under appropriate nonresonant conditions. His results was partially generalized to infinite-dimensional setting in~\cite{E2,R-S-JDDE04,ZZJ}. For other relevant results in this direction, we refer to~\cite{E1,R-S-JDDE06, ZLZ, ZhangZhang14JDE} and references therein.

\subsection{Nonautonomous topological linearization}
So far, we have only discussed the case of autonomous  dynamics. The first version of the Grobman-Hartman theorem for nonautonomous dynamics was obtained by Palmer~\cite{Palmer}. Let us formulate the version of his result in the case of discrete time which was obtained in~\cite{AW}. Assume that $(A_n)_{n\in \Z}$ is a sequence of bounded and invertible linear operators acting on  a Banach space $X=(X, |\cdot |)$. Furthermore, suppose that $f_n\colon X \to X$, $n\in \Z$ is a sequence of (nonlinear) maps. We can consider the associated nonlinear and nonautonomous difference equation
\begin{equation}\label{i1}
x_{n+1}=A_n x_n+f_n(x_n) \quad n\in \Z,
\end{equation}
 as well as its linear part given by
\begin{equation}\label{i2}
x_{n+1}=A_n x_n \quad n\in \Z.
\end{equation}
Assume that the following conditions hold:
\begin{itemize}
\item the sequence $(A_n)_{n\in \Z}$ admits an exponential dichotomy (see~\cite{C});
\item there exist $M,c>0$ such that $|f_n(x)| \le M$ and 
\[
|f_n(x)-f_n(y)| \le c|x-y|, 
\]
for $n\in \Z$ and $x, y\in X$.
\end{itemize}
Then, provided that $c$ is sufficiently small, equations~\eqref{i1} and~\eqref{i2} are topologically equivalent. Moreover, under additional mild conditions the conjugacies are locally H\"{o}lder continuous~\cite{SX}.

More recently, several authors obtained extensions of Palmer's theorem by weakening the hyperbolicity requirement for the linear part~\eqref{i2}, which also required changing the assumptions dealing with nonlinear terms $f_n$. We refer to~\cite{BD,Jiang, Jiang2, Lin, Reinf,  RS} and references therein.

\subsection{Nonautonomous linearization: higher regularity of conjugacies}
Motivated by previously described results dealing with autonomous dynamics, in recent years several authors dealt with the problem of formulating sufficient conditions under which the conjugacies in the nonautonomous linearization exhibit higher regularity.

To the best of our knowledge, the first results in this direction are those in~\cite{CR} (see also~\cite{CMR}) dealing with the case when the linear part admits an exponential contraction. In~\cite{CDS}, the authors have established the nonautonomous versions of Sternberg's results by formulating sufficient conditions under which the conjugacies are of order $C^r$, $r\ge 1$. These conditions require that the linear part~\eqref{i2} admits an exponential dichotomy and that appopriate nonresonant conditions expressed in terms of the so-called Sacker-Sell spectrum hold.
Similar in nature (although different in techniques) are the result established in~\cite{DZZ,DZZ2} which give conditons for $C^1$-linearization of nonautonomous systems. There, it is required that~\eqref{i2} admits a nonuniform exponential dichotomy and that appropriate spectral-gap conditions hold. 

Finally, we mention the recent paper~\cite{CJ} in which the authors formulate sufficient conditions for $C^1$-linearization of nonautonomous discrete systems. The main novelties when compared with the previous works are the following:
\begin{itemize}
\item \eqref{i2} is assumed to admit a nonuniform dichotomy which is not necessarily exponential;
\item no nonresonance or spectral-gap conditions are required.
\end{itemize}
The main idea of~\cite{CJ} is that the lack of hyperbolicity and spectral data can be compensated by appropriate assumptions related to the size of nonlinear terms $f_n$ in~\eqref{i1}.
However, it should be emphasized that when restricted to the setting of (for example)~\cite{DZZ}, the result in~\cite{CJ} can fail to be applicable or could yield weaker results from that in~\cite{DZZ}. We refer to Remark~\ref{EMO} for a detailed discussion.

\subsection{Contributions of the present paper}
In the present paper (following~\cite{BDP,SX,XWKR}),  we study the coupled nonautonomous and nonlinear system
\begin{equation}\label{i3}
x_{n+1} =A_nx_n+f_n(x_n, y_n), \quad y_{n+1}=g_n( y_n),
\end{equation}
as well as the associate uncoupled system
\begin{equation}\label{i4}
x_{n+1} =A_nx_n, \quad y_{n+1}=g_n( y_n).
\end{equation}
Here, $(A_n)_{n\in \Z}$ is a sequence of bounded and invertible linear operators on a Banach space $X$, $g_n\colon Y \to Y$, $n\in \Z$ is a sequence of homeomorphisms of a Banach space $Y$, and $f_n\colon X\times Y \to X$, $n\in \Z$ is a sequence of nonlinear maps.
Observe that when $Y$ is a trivial Banach space, \eqref{i3} and~\eqref{i4} reduce to~\eqref{i1} and~\eqref{i2}, respectively.

In~\cite{BDP}, the authors have formulated very general conditions under which~\eqref{i3} and~\eqref{i4} are topologically conjugated. The goal of the present paper is to formulate conditions under which these conjugacies are smooth. Our approach is inspired by that in the previously described paper by Casta\~{n}eda and Jara~\cite{CJ}. We stress that our arguments require a nontrivial modification of those in~\cite{CJ} as those need to be combined with techniques developed in~\cite{BDP}. In addition, we do not impose any dichotomy conditions for the linear part in~\eqref{i4}.

\section{Preliminaries}\label{sec: DT} 
Let $(X, |\cdot |_X)$ and $(Y, |\cdot |_Y)$ be two arbitrary Banach spaces. For the sake of simplicity both norms $|\cdot |_X$ and $|\cdot |_Y$ will be denoted simply by $|\cdot |$. By $\mathcal B(X)$ we denote the space of all bounded operators on $X$ equipped with the operator norm, which we will also denote by $|\cdot |$.

Let $(A_n)_{n\in \Z}$ be a sequence of invertible operators in $\mathcal B(X)$ and  $f_n \colon X\times Y \to X$, $n\in \Z$  a sequence of maps. Furthermore, let $(g_n)_{n\in \Z}$ be a sequence of homeomorphisms on $Y$. We consider the associated coupled system 
\begin{equation}\label{nnd}
x_{n+1} =A_nx_n+f_n(x_n, y_n), \quad y_{n+1}=g_n( y_n)
\end{equation}
as well the  uncoupled system
\begin{equation}\label{lnd}
x_{n+1}=A_n x_n, \quad y_{n+1}=g_n(y_n).
\end{equation}
Let  $k\mapsto (x_1(k, n, x), y(k, n, y))$ denote the solution of~\eqref{lnd} which equals $(x, y)$ when $k=n$. Similarly, let $k\mapsto (x_2(k, n, x, y), y(k, n, y))$ be  the solution of~\eqref{nnd}
with value  $(x, y)$ at $k=n$.

For $m, n\in \Z$, set 
\[
\mathcal A(m, n)=\begin{cases}
A_{m-1}\cdots A_n & \text{for $m>n$;}\\
\Id  &\text{for $m=n$;} \\
A_m^{-1}\cdots A_{n-1}^{-1} & \text{for $m<n$,}
\end{cases}
\]
where $\Id$ denotes the identity operator on $X$.
Let $(P_n)_{n\in \Z}$ be a sequence in $\mathcal B(X)$. We emphasize that the maps $P_n$ do not need to be projections. Then we define 
\begin{equation}\label{G-d}
\mathcal G(m, n)=\begin{cases}
\mathcal A(m, n)P_n & \text{for $m\ge n$;}\\
-\mathcal A(m, n)(\Id-P_n) & \text{for $m<n$.}
\end{cases}
\end{equation}

\subsection{Hypothesis} In order to facilitate future reference, we gather in this subsection all the hypothesis we are going to use in our results. We group them into two categories: basic and advanced conditions. 

\subsubsection{Basic conditions:}
\begin{enumerate}[label={(BC\arabic*)},ref=(BC\arabic*)]
\item \label{H: fn is Lipschitz and bounded}
there exist sequences $(\mu_n)_{n\in \Z}$ and $(\gamma_n)_{n\in \Z}$ in $[0, \infty)$ such that 
\begin{equation*}
|f_n(x, y)|\le \mu_n \quad \text{and} \quad |f_n(x_1, y)-f_n(x_2, y)|\le \gamma_n |x_1-x_2|,
\end{equation*}
for $n\in \Z$, $x, x_1, x_2\in X$ and $y\in Y$;

\item \label{H: N finite}
\begin{equation*}
N:=\sup_{m\in \Z} \sum^{\infty}_{n=-\infty}|\mathcal G(m,n)|\mu_{n-1} <\infty;
\end{equation*}
\item \label{H: q smaller 1}
\begin{equation*}
q:= \sup_{m\in \Z}\sum^{\infty}_{n=-\infty}|\mathcal G(m,n)|\gamma_{n-1} <1;
\end{equation*}

\item \label{H: Aj gamma j smaller 1} $|A_n^{-1}|\gamma_n<1$ for every $n\in \Z$.
\end{enumerate}

The basic conditions~\ref{H: fn is Lipschitz and bounded}, \ref{H: N finite}, \ref{H: q smaller 1} and~\ref{H: Aj gamma j smaller 1} ensure that the systems~\eqref{i3} and~\eqref{i4} are topologically equivalent (note that condition \ref{H: Aj gamma j smaller 1} guarantees that the solutions of~\eqref{i3} are defined globally). More precisely, we have the following result.

\begin{theorem}[Theorem 3.1 of \cite{BDP}]\label{theo: BDP disc}
Suppose that \ref{H: fn is Lipschitz and bounded}, \ref{H: N finite}, \ref{H: q smaller 1} and~\ref{H: Aj gamma j smaller 1} are satisfied. Then, there exists a sequence of continuous functions $H_n\colon X\times Y\to X\times Y$, $n\in \Z$ of the form $H_n(x, y)=(x+h_n(x, y), y)$, where  $\sup_{n,x,y}|h_n(x, y)|<\infty$, such that if $n\mapsto (x_n,y_n)$ is a solution of (\ref{lnd}),
then $n\mapsto H_n(x_n, y_n)$ is a solution of (\ref{nnd}). In addition, there exists a sequence of continuous functions $\bar H_n \colon X\times Y\to X\times Y$, $n\in \Z$ of the form $\bar H_n( x, y)=(x+\bar h_n(x, y), y)$, where $\sup_{n,x, y}|\bar h_n(x, y)|<\infty$,  such that if $n\mapsto (x_n, y_n)$ is a solution of (\ref{nnd}),
then $n\mapsto \bar H_n(x_n, y_n)$ is a solution of (\ref{lnd}). Moreover, $H_n$ and $\bar H_n$ are inverses of
each other, that is,
\begin{equation*}
H_n(\bar H_n(x,y))=(x,y)=\bar H_n(H_n(x,y)),
\end{equation*}
for $n\in \Z$ and $(x, y)\in X\times Y$.
\end{theorem}
Furthermore, under some additional assumptions it was proved in \cite{BDP} that the conjugacies $H_n(x,y)$ and $\bar H_n(x,y)$ are H\"older continuous with respect to the variables $x$ and $y$. Our objective now is to formulate sufficient conditions under which those maps exhibit higher regularity properties. For this purpose we introduce the next set of conditions.

\subsubsection{Advanced conditions:}
\begin{enumerate}[label={(AC\arabic*)},ref=(AC\arabic*)]

\item \label{H: fn diff first var} for every $y\in Y$, the map $x\to f_n(x,y)$ is $C^r$, $r\geq 1$, for every $n\in \Z$;

\item \label{H: Kn+Jn finite} given $n\in \Z$,
\begin{displaymath}
K_n:=\sum _{k< n} |\mathcal{G} (n,k+1)|\gamma_k \left(\prod_{j=k}^{n-1} \frac{|A_j^{-1}|}{1-\gamma_j|A_j^{-1}|}\right) <+\infty
\end{displaymath}
and 
\begin{displaymath}
J_n:=\sum _{k> n} |\mathcal{G} (n,k+1)|\gamma_k \left(\prod_{j=n}^{k-1} (|A_j|+\gamma_j)\right)<+\infty;
\end{displaymath}

\item \label{H: Kn+Jn smaller 1} given $n\in \Z$,
\begin{displaymath}
K_n+J_n+|\mathcal{G} (n,n+1)|\gamma_n <1;
\end{displaymath}

\item \label{H: fn is Lipschitz second var} there exists a sequence $(\rho_n)_{n\in \Z}$ in $[0, \infty)$ such that 
\begin{equation*}
|f_n(x, y_1)-f_n(x, y_2)|\le \rho_n |y_1-y_2|,
\end{equation*}
for $n\in \Z$, $x\in X$ and $y_1,y_2\in Y$;

\item \label{H: gn is Lipschitz} 
there exist sequences $(\tau_n)_{n\in \Z}$ and $(\sigma_n)_{n\in \Z}$ in $[0, \infty)$ such that 
\begin{equation*}
|g_n(y_1)-g_n(y_2)|\le \tau_n |y_1-y_2| \text{ and } |g^{-1}_n(y_1)-g^{-1}_n(y_2)|\le \sigma_n |y_1-y_2|,
\end{equation*}
for every $n\in \Z$ and $y_1,y_2\in Y$.

\item \label{H: sigma rho smaller than 1}
$\sigma_n\rho_n\leq 1$ for every $n\in \Z$;

\item\label{H: fn diff second var} the map $y\to f_n(x,y)$ is $C^r$, $r\geq 1$, for every $n\in \Z$;

\item \label{H: gn diff} the map $y\to g_n(y)$ is a $C^r$-diffeomorphism, $r\geq 1$, for every $n\in \Z$;

\item \label{H: temporary name}
given $n\in \Z$, 
\begin{displaymath}
\sum_{k\in \Z}| \mathcal G(n, k+1)|\left(\gamma_k M_{k,n}+\rho_k D_{k,n}\right) <+\infty
\end{displaymath}
where $D_{k,n}$ and $M_{k, n}$ are given by~\eqref{Dkn} and~\eqref{Mkn} respectively.

\end{enumerate}

\begin{remark} The advanced conditions \ref{H: fn diff first var}-\ref{H: temporary name} will play a fundamental role in improving the regularity of the topological conjugacies given by Theorem \ref{theo: BDP disc}. More precisely,
\begin{itemize}
\item provided that~\ref{H: fn diff first var} holds and that~\ref{H: Kn+Jn finite} is satisfied for $n\in \Z$, we have that the conjugacy $\bar H_n$ given by Theorem~\ref{theo: BDP disc} is $C^1$ in  the first variable (see Theorem~\ref{theo: reg H bar}). In addition, we can ensure that~\ref{H: Kn+Jn finite} holds simultaneously for every $n\in \Z$ (see Example~\ref{remm});
\item provided that~\ref{H: fn diff first var} holds and that~\ref{H: Kn+Jn smaller 1} is satisfied for $n\in \Z$, we have that the conjugacy $H_n$ given by Theorem~\ref{theo: BDP disc} is $C^1$ in the first variable  (see Theorem~\ref{theo: reg H}). In addition, we can ensure that~\ref{H: Kn+Jn smaller 1} holds simultaneously for every $n\in \Z$ (see Examples~\ref{EX1} and~\ref{EX2});
\item if~\ref{H: fn diff first var}, \ref{H: fn is Lipschitz second var}, \ref{H: gn is Lipschitz}, \ref{H: sigma rho smaller than 1},  \ref{H: fn diff second var}, \ref{H: gn diff} hold and~\ref{H: temporary name} is satisfied for $n\in \Z$, then $\bar H_n$ is $C^1$ in the second variable (see Theorem~\ref{theo: reg H bar second var});
\item suppose that~\ref{H: fn diff first var}, \ref{H: fn is Lipschitz second var}, \ref{H: gn is Lipschitz}, \ref{H: sigma rho smaller than 1}, \ref{H: fn diff second var} and \ref{H: gn diff} hold and that~\ref{H: Kn+Jn smaller 1}  and~\ref{H: temporary name} are satisfied for $n\in \Z$. Then, $H_n$ is $C^1$ in the second variable (see Theorem~\ref{theo: reg H second var}). In addition, we can ensure that~\ref{H: temporary name} holds simultaneously for every $n\in \Z$ (see Example~\ref{end}).
\end{itemize}
\end{remark}

\section{Main results}

\subsection{Regularity of conjugacies with respect to the first variable} \label{sec: reg first var}
We are interested in formulating sufficient conditions under which conjugacies $H_n$ and $\bar H_n$ are smooth with respect to the first variable.

We start by studying the regularity properties of the map $\xi \to x_2(k,n,\xi,\eta)$.
Consider
\[
 C_{k,n}=\begin{cases}
\prod _{j=n}^{k-1}(|A_j|+\gamma_j) & \text{for $k>n$;}\\
1  &\text{for $k=n$;} \\
\prod_{j=k}^{n-1} \frac{|A_j^{-1}|}{1-\gamma_j|A_j^{-1}|} & \text{for $k<n$.}
\end{cases}
\]

\begin{lemma} \label{lemma: sol lips}
Suppose \ref{H: fn is Lipschitz and bounded} and \ref{H: Aj gamma j smaller 1} are satisfied. Then, for any $\xi,\zeta \in X$, $\eta \in Y$ and $n,k\in\Z$, we have that 
\begin{displaymath}
|x_2(k,n,\xi,\eta)-x_2(k,n,\zeta,\eta)|\leq C_{k,n}|\xi - \zeta|.
\end{displaymath}
\end{lemma}

\begin{proof}
Fix $n\in \Z$. By \eqref{nnd} we have that
\begin{equation}\label{eq: x2 future}
\begin{split}
x_2(n+1,n,\xi,\eta)&=A_nx_2(n,n,\xi,\eta)+f_n(x_2(n,n,\xi,\eta),y(n,n,\eta))\\
&=A_n\xi+f_n(\xi,\eta).\\
\end{split}
\end{equation}
Consequently, using \ref{H: fn is Lipschitz and bounded},
\begin{displaymath}
\begin{split}
|x_2(n+1,n,\xi,\eta)-x_2(n+1,n,\zeta,\eta)|&=|A_n\xi+f_n(\xi,\eta)-A_n\zeta-f_n(\zeta,\eta)|\\
&\leq \left(|A_n|+\gamma_n\right)|\xi-\zeta|.
\end{split}
\end{displaymath}
Thus, using that
\begin{displaymath}
x_2(k,n,\xi,\eta)=x_2(k,m,x_2(m,n,\xi,\eta),y(m,n,\eta)) \quad k\ge m\ge n,
\end{displaymath} 
and proceeding inductively we conclude that 
\begin{displaymath}
\begin{split}
|x_2(k,n,\xi,\eta)-x_2(k,n,\zeta,\eta)|&\leq \prod _{j=n}^{k-1}(|A_j|+\gamma_j)|\xi - \zeta|\\
&= C_{k,n}|\xi - \zeta|,
\end{split}
\end{displaymath}
for every $k> n$.

We now consider the case when $k<n$. Given $(\xi,\eta)\in X\times Y$ and $j\in \Z$, let us consider
\begin{displaymath}
F_j(\xi,\eta)=A_j\xi +f_j(\xi,\eta)
\end{displaymath}
and 
\begin{equation}\label{Tj}
T_j(\xi,\eta)=A_j^{-1}\xi-A_j^{-1}f_j(T_j(\xi,\eta),\eta).
\end{equation}
It is easy to check that $F_j(T_j(\xi,\eta),\eta)=\xi$. In addition, we we have that
\[
\begin{split}
T_j(F_j(\xi,\eta),\eta) &=A_j^{-1}F_j(\xi,\eta)-A_j^{-1}f_j(T_j(F_j(\xi,\eta),\eta) )\\
&=A_j^{-1}(A_j\xi +f_j(\xi,\eta))-A_j^{-1}f_j(T_j(F_j(\xi,\eta),\eta) )\\
&=\xi+A_j^{-1}f_j(\xi,\eta)-A_j^{-1}f_j(T_j(F_j(\xi,\eta),\eta)),
\end{split}
\]
which implies that
\[
|T_j(F_j(\xi,\eta),\eta) -\xi| \le |A_j^{-1}| \gamma_j|T_j(F_j(\xi,\eta),\eta) -\xi|.
\]
 By~\ref{H: Aj gamma j smaller 1}, we conclude that $T_j(F_j(\xi,\eta),\eta)=\xi$. Moreover, since
\begin{displaymath}
\begin{split}
\xi&=A_{n-1}x_2(n-1,n,\xi,\eta)+f_{n-1}(x_2(n-1,n,\xi,\eta),y(n-1,n,\eta))\\
&=F_{n-1}(x_2(n-1,n,\xi,\eta),y(n-1,n,\eta)),
\end{split}
\end{displaymath}
it follows from the previous observations that
\begin{equation}\label{eq: Tn x2}
T_{n-1}(\xi,y(n-1,n,\eta))=x_2(n-1,n,\xi,\eta).
\end{equation}
Consequently,
\begin{displaymath}
\begin{split}
&|x_2(n-1,n,\xi,\eta)-x_2(n-1,n,\zeta,\eta)| \\
&=|T_{n-1}(\xi,y(n-1,n,\eta))-T_{n-1}(\zeta,y(n-1,n,\eta))|.
\end{split}
\end{displaymath}
Now, using \ref{H: fn is Lipschitz and bounded}, for any $y\in Y$ we have that
\begin{displaymath}
\begin{split}
&|T_{n-1}(\xi,y)-T_{n-1}(\zeta,y)| \\
&=|A_{n-1}^{-1}(\xi-\zeta)+A_{n-1}^{-1}\left(f_{n-1}(T_{n-1}(\zeta,y),y)-f_{n-1}(T_{n-1}(\xi,y),y)\right)|\\
&\leq |A_{n-1}^{-1}||\xi-\zeta|+|A_{n-1}^{-1}|\gamma_{n-1}|T_{n-1}(\zeta,y)-T_{n-1}(\xi,y)|.
\end{split}
\end{displaymath}
Therefore, using \ref{H: Aj gamma j smaller 1} we conclude that
\begin{displaymath}
\begin{split}
|T_{n-1}(\xi,y)-T_{n-1}(\zeta,y)|&\leq \frac{|A_{n-1}^{-1}|}{1-\gamma_{n-1}|A_{n-1}^{-1}|} |\xi-\zeta|,
\end{split}
\end{displaymath}
which combined with the previous observations implies that
\begin{displaymath}
\begin{split}
|x_2(n-1,n,\xi,\eta)-x_2(n-1,n,\zeta,\eta)|&\leq \frac{|A_{n-1}^{-1}|}{1-\gamma_{n-1}|A_{n-1}^{-1}|} |\xi-\zeta|.
\end{split}
\end{displaymath}
 Consequently, by  proceeding inductively as in the case when $k>n$,  we conclude that 
\begin{displaymath}
\begin{split}
|x_2(k,n,\xi,\eta)-x_2(k,n,\zeta,\eta)|&\leq  C_{k,n}|\xi - \zeta|
\end{split}
\end{displaymath}
for every $k< n$. The case when $k=n$ is obvious and therefore the proof is complete.
\end{proof}

\begin{proposition}\label{prop: diff of sol}
Assume that \ref{H: fn is Lipschitz and bounded}, \ref{H: Aj gamma j smaller 1} and \ref{H: fn diff first var} hold. Then, the map $\xi \to x_2(k,n,\xi,\eta)$ is of class $C^r$, $r\geq 1$, for every $n,k\in\mathbb{Z}$ and $\eta \in Y$.
\end{proposition}
\begin{proof}
Fix $n\in \Z$. By \eqref{nnd},  we have that
\begin{displaymath}
\begin{split}
x_2(n+1,n,\xi,\eta)&=A_n\xi+f_n(\xi,\eta).\\
\end{split}
\end{displaymath}
Consequently, it follows from~\ref{H: fn diff first var} that $\xi \to x_2(n+1,n,\xi,\eta)$ is of class $C^r$. Thus, proceeding inductively, we conclude that $\xi \to x_2(k,n,\xi,\eta)$ is of class $C^r$ for every $k>n$.

Next, we claim that the map $\xi \to T_{j}(\xi,\eta)$ is of class $C^r$, where $T_j$ is given by~\eqref{Tj}. Indeed, by \ref{H: fn is Lipschitz and bounded} we have that
\begin{equation}\label{eq: bound dif of f}
\begin{split}
\left|\frac{\partial f_j(u,v)}{\partial u}\right|&=\lim_{\delta \to 0} \frac{|f_j(u+\delta,v)-f_j(u,v)|}{|\delta|}\\
&\leq \lim_{\delta \to 0} \frac{\gamma_j|\delta|}{|\delta|}\\
&=\gamma_j,
\end{split}
\end{equation}
where $\frac{\partial f_j(\cdot,\cdot)}{\partial u}$ denotes the derivative of $f_j$ with respect to the first variable.
Thus,  by \ref{H: Aj gamma j smaller 1}, it follows that the operator $\Id+A_j^{-1}\frac{\partial f_j(\xi,\eta)}{\partial \xi}$ is invertible. In particular, we can consider
\begin{equation}\label{eq: def of L}
L: =\left(\Id+A_j^{-1}\frac{\partial f_j(T_j(\xi,\eta),\eta)}{\partial u}\right)^{-1}A_j^{-1}=\left(A_j+\frac{\partial f_j(T_j(\xi,\eta),\eta)}{\partial u}\right)^{-1}.
\end{equation}
 Now, observing that
\begin{displaymath}
\begin{split}
&T_{j}(\xi+\delta,\eta)-T_{j}(\xi,\eta)-L\delta \\
&=A_j^{-1}\delta-A_j^{-1}\left(f_j(T_j(\xi+\delta,\eta),\eta)-f_j(T_j(\xi,\eta),\eta)\right)-L\delta\\
&=A_j^{-1}\delta-A_j^{-1}\left(f_j(T_j(\xi+\delta,\eta),\eta)-f_j(T_j(\xi,\eta)+L\delta,\eta)\right)\\
&\phantom{=}-A_j^{-1}\left(f_j(T_j(\xi,\eta)+L\delta,\eta)-f_j(T_j(\xi,\eta),\eta)\right)-L\delta,\\
\end{split}
\end{displaymath}
it follows from \ref{H: fn is Lipschitz and bounded} that
\begin{displaymath}
\begin{split}
&\left| T_{j}(\xi+\delta,\eta)-T_{j}(\xi,\eta)-L\delta\right| \\
&\leq \left|A_j^{-1}\left(f_j(T_j(\xi+\delta,\eta),\eta)-f_j(T_j(\xi,\eta)+L\delta,\eta)\right)\right|\\
&\phantom{=}+\left|A_j^{-1}\delta-A_j^{-1}\left(f_j(T_j(\xi,\eta)+L\delta,\eta)-f_j(T_j(\xi,\eta),\eta)\right)-L\delta\right|\\
&\leq \gamma_j|A_j^{-1}| \left|T_j(\xi+\delta,\eta)-T_j(\xi,\eta)-L\delta\right|\\
&\phantom{=}+\left|A_j^{-1}\delta-A_j^{-1}\left(f_j(T_j(\xi,\eta)+L\delta,\eta)-f_j(T_j(\xi,\eta),\eta)\right)-L\delta\right|.\\
\end{split}
\end{displaymath}
Using~\ref{H: Aj gamma j smaller 1}, we obtain that
\begin{displaymath}
\begin{split}
&(1-\gamma_j|A_j^{-1}|)\left| T_{j}(\xi+\delta,\eta)-T_{j}(\xi,\eta)-L\delta\right| \\
&\leq \left|A_j^{-1}\delta-A_j^{-1}\left(f_j(T_j(\xi,\eta)+L\delta,\eta)-f_j(T_j(\xi,\eta),\eta)\right)-L\delta\right|.\\
\end{split}
\end{displaymath}
Therefore, since by \ref{H: fn diff first var} $u\to f_j(T_j(\xi,\eta)+Lu,\eta) $ is differentiable,
\begin{displaymath}
\frac{\partial f_j(T_j(\xi,\eta)+Lu,\eta)}{\partial u}\bigg |_{u=0} =\frac{\partial f_j(T_j(\xi,\eta),\eta)}{\partial u} L
\end{displaymath}
and since 
\[
\begin{split}
&A_j^{-1}\frac{\partial f_j(T_j(\xi,\eta),\eta)}{\partial u}L +L \\
&=\left(\Id+A_j^{-1}\frac{\partial f_j(T_j(\xi,\eta),\eta)}{\partial u}\right )L\\
&=\left(\Id+A_j^{-1}\frac{\partial f_j(T_j(\xi,\eta),\eta)}{\partial u}\right )\left(\Id+A_j^{-1}\frac{\partial f_j(T_j(\xi,\eta),\eta)}{\partial u}\right)^{-1}A_j^{-1}\\
&=A_j^{-1},
\end{split}
\]
we get that
\begin{displaymath}
\begin{split}
\lim_{\delta\to 0} \frac{\left|A_j^{-1}\delta-A_j^{-1}\left(f_j(T_j(\xi,\eta)+L\delta,\eta)-f_j(T_j(\xi,\eta),\eta)\right)-L\delta\right|}{|\delta|}=0.\\
\end{split}
\end{displaymath}
Thus, combining the previous observations we conclude that
\begin{displaymath}
\begin{split}
\lim_{\delta\to 0}\frac{\left| T_{j}(\xi+\delta,\eta)-T_{j}(\xi,\eta)-L\delta \right|}{|\delta|}=0.\\
\end{split}
\end{displaymath}
In particular, $T_j(\xi,\eta)$ is differentiable with respect to $\xi$ and $\frac{\partial T_j(\xi,\eta)}{\partial \xi} =L$. Moreover, from~\eqref{eq: def of L} and \ref{H: fn diff first var}, we have that  $\xi \to T_j(\xi,\eta)$ is $C^r$. 

Consequently, since by \eqref{eq: Tn x2} we have that $T_{n-1}(\xi,y(n-1,n,\eta))=x_2(n-1,n,\xi,\eta)$, it follows that $\xi \to x_2(n-1,n,\xi,\eta)$ is $C^r$. Finally, proceeding inductively we conclude that $\xi \to x_2(k,n,\xi,\eta)$ is $C^r$ for every $k<n$. Therefore, since the desired conclusion clearly holds for $k=n$,  the proof of the proposition is completed.
\end{proof}

\begin{theorem}\label{theo: reg H bar}
Suppose that \ref{H: fn is Lipschitz and bounded}, \ref{H: N finite}, \ref{H: q smaller 1}, \ref{H: Aj gamma j smaller 1} and \ref{H: fn diff first var} hold. Moreover, given $n\in \Z$ suppose that hypothesis \ref{H: Kn+Jn finite} is satisfied for $n$. Then, the map $\xi \to \bar{H}_n(\xi,\eta)$ is $C^1$.
\end{theorem}
\begin{proof}
Fix $n\in \Z$. We start by recalling (see the proof of \cite[Theorem 3.1]{BDP}) that 
\[ \bar H_n(\xi,\eta)=(\xi+\bar h_n(\xi,\eta),\eta),\]
where
\begin{equation}\label{eq: hn bar}
\bar h_n(\xi, \eta)=-\sum_{k\in \Z} \mathcal G(n, k+1)f_{k}(x_2(k, n, \xi, \eta), y(k, n, \eta)),
\end{equation}
for $(\xi, \eta)\in X\times Y$. Then, it remains to prove that  the map $\xi\to \bar h_n(\xi,\eta)$ is $C^1$.

By hypothesis \ref{H: fn diff first var} and Proposition \ref{prop: diff of sol}, we have that each of the terms involved in the series \eqref{eq: hn bar} is differentiable with respect to $\xi$. Moreover, by \eqref{eq: bound dif of f} we have that
\begin{displaymath}
\begin{split}
\left|\frac{\partial f_k(u,v)}{\partial u}\right|&\leq \gamma_k.
\end{split}
\end{displaymath}
Similarly, using Lemma \ref{lemma: sol lips} we conclude that
\begin{displaymath}
\begin{split}
\left|\frac{\partial x_2(k, n, \xi, \eta)}{\partial \xi}\right|&=\lim_{\delta\to 0}\frac{|x_2(k, n, \xi+\delta, \eta)-x_2(k, n, \xi, \eta)|}{|\delta|}\\
&\leq \lim_{\delta\to 0}\frac{C_{k,n}|\delta|}{|\delta|}\\
& =C_{k,n}.
\end{split}
\end{displaymath}
Combining these observations with \ref{H: Kn+Jn finite}, it follows that
\begin{equation}\label{17}
\begin{split}
&\sum_{k\in \Z} \left| \frac{\partial}{\partial \xi}  \mathcal G(n, k+1)f_{k}(x_2(k, n, \xi, \eta), y(k, n, \eta))\right| 
\\
&\leq \sum_{k\in \Z}| \mathcal G(n, k+1)|\left| \frac{\partial f_{k}(x_2(k, n, \xi, \eta), y(k, n, \eta))}{\partial u}\right| \left|\frac{\partial x_2(k, n, \xi, \eta)}{\partial \xi}\right|\\
&\leq \sum_{k\in \Z}| \mathcal G(n, k+1)|\gamma_k C_{k,n}<+\infty.\\
\end{split}
\end{equation}
In particular, the series
\begin{displaymath}
-\sum_{k\in \Z}  \frac{\partial}{\partial \xi}  \mathcal G(n, k+1)f_{k}(x_2(k, n, \xi, \eta), y(k, n, \eta))
\end{displaymath}
converges uniformly. Thus, it coincides with $\frac{\partial \bar h_n(\xi, \eta) }{\partial \xi}$ and $\xi\to \bar h_n(\xi,\eta)$ is $C^1$. Consequently, $\xi\to \bar H_n(\xi,\eta)$ is $C^1$ as claimed.
\end{proof}

\begin{theorem}\label{theo: reg H}
Suppose hypothesis \ref{H: fn is Lipschitz and bounded}, \ref{H: N finite}, \ref{H: q smaller 1}, \ref{H: Aj gamma j smaller 1} and \ref{H: fn diff first var} are satisfied. Moreover, given $n\in \Z$ assume that the  hypothesis \ref{H: Kn+Jn smaller 1} is satisfied for $n$. Then, the map $\xi \to H_n(\xi,\eta)$ is $C^1$.
\end{theorem}
\begin{proof}
Recall that by Theorem \ref{theo: BDP disc}, $H_n(\xi, \eta)=(\xi+h_n(\xi, \eta), \eta)$. Hence, it remains to prove  that the map $\xi\to h_n(\xi, \eta)$ is $C^1$. We start by observing that hypothesis \ref{H: Kn+Jn smaller 1} and the estimates from the proof of Theorem \ref{theo: reg H bar} (see~\eqref{17}) imply that $\left|\frac{\partial \bar h_n(\xi, \eta) }{\partial \xi}\right|<1$, for every $(\xi,\eta)\in X\times Y$. In particular, the operator $\Id +\frac{\partial \bar h_n(\xi, \eta) }{\partial \xi}$ is invertible and we may consider
\begin{equation}\label{eq: def of R}
R:=-\left(\Id +\frac{\partial \bar h_n(\xi +h_n(\xi, \eta), \eta) }{\partial u}\right)^{-1}\frac{\partial \bar h_n(\xi +h_n(\xi, \eta),\eta) }{\partial u},
\end{equation}
where $\frac{\partial \bar h_n(\cdot,\cdot) }{\partial u}$ denotes the derivative of $\bar h_n(\cdot,\cdot)$ with respect to the first variable. 

We  now claim that $\xi\to h_n(\xi, \eta)$ is differentiable and that $\frac{\partial h_n(\xi, \eta) }{\partial \xi}=R$. We first note that Theorem~\ref{theo: BDP disc} implies that  $\bar{H}_n(H_n(\xi,\eta))=(\xi,\eta)$, and consequently 
\begin{equation}\label{2131}
h_n(\xi,\eta)=-\bar h_n(\xi+h_n(\xi,\eta),\eta), \quad \text{for every $(\xi,\eta)\in X\times Y$.}
\end{equation}
 Moreover, since by Theorem~\ref{theo: reg H bar} $\bar{h}_n$ is differentiable,   we have that
\begin{displaymath}
\begin{split}
&h_n(\xi+\delta,\eta)-h_n(\xi,\eta)\\
&=-\bar h_n(\xi+\delta+h_n(\xi+\delta,\eta),\eta)+\bar h_n(\xi+h_n(\xi,\eta),\eta)\\
&=-\frac{\partial\bar h_n(\xi+h_n(\xi,\eta),\eta)}{\partial u}\left(\delta+h_n(\xi+\delta,\eta)-h_n(\xi,\eta)\right)\\
&\phantom{=}+\mathbf{o}(\delta+h_n(\xi+\delta,\eta)-h_n(\xi,\eta)),\\
\end{split}
\end{displaymath}
where $l(x)=\mathbf{o}(x)$ means that $\lim_{x \to 0}\frac{|l(x)|}{|x|}=0$. Thus,
\begin{displaymath}
\begin{split}
&\left(\Id+\frac{\partial\bar h_n(\xi+h_n(\xi,\eta),\eta)}{\partial u}\right) \left( h_n(\xi+\delta,\eta)-h_n(\xi,\eta) \right)\\
&=-\frac{\partial\bar h_n(\xi+h_n(\xi,\eta),\eta)}{\partial u}\delta\\
&\phantom{=}+\mathbf{o}(\delta+h_n(\xi+\delta,\eta)-h_n(\xi,\eta)),\\
\end{split}
\end{displaymath}
which implies that
\begin{equation}\label{eq: est hn o}
\begin{split}
h_n(\xi+\delta,\eta)-h_n(\xi,\eta) =R\delta+\mathbf{o}(\delta+h_n(\xi+\delta,\eta)-h_n(\xi,\eta)).\\
\end{split}
\end{equation}
On the other hand, since $\xi \to h_n(\xi,\eta)$ is continuous,
\begin{displaymath}
\mathbf{o}(\delta+h_n(\xi+\delta,\eta)-h_n(\xi,\eta))\leq \frac{1}{2}|\delta +h_n(\xi+\delta,\eta)-h_n(\xi,\eta)|,
\end{displaymath}
whenever $|\delta|$ is sufficiently small. Consequently, by  taking the norm on both sides of \eqref{eq: est hn o} we get that
\begin{equation*}
\begin{split}
|h_n(\xi+\delta,\eta)-h_n(\xi,\eta)|&\leq |R||\delta|+\frac{1}{2}|\delta +h_n(\xi+\delta,\eta)-h_n(\xi,\eta)|\\
&\leq \left(|R|+\frac{1}{2}\right)|\delta|+\frac{1}{2}|h_n(\xi+\delta,\eta)-h_n(\xi,\eta)|,\\
\end{split}
\end{equation*}
whenever $|\delta|$ is sufficiently small. In particular, there exist $C>0$ and $\rho>0$ such that
\begin{displaymath}
\frac{|h_n(\xi+\delta,\eta)-h_n(\xi,\eta)|}{|\delta|}\leq C,  \quad \text{for  $|\delta|<\rho$.}
\end{displaymath}
Thus, whenever $|\delta|<\rho$,  we have that
\begin{equation}\label{eq: final eq est o hn}
\begin{split}
&\frac{\mathbf{o}(\delta+h_n(\xi+\delta,\eta)-h_n(\xi,\eta))}{|\delta|}\\
&=\frac{\mathbf{o}(\delta+h_n(\xi+\delta,\eta)-h_n(\xi,\eta))}{|\delta+h_n(\xi+\delta,\eta)-h_n(\xi,\eta)|}\cdot
\frac{|\delta +h_n(\xi+\delta,\eta)-h_n(\xi,\eta)|}{|\delta|}\\
&\leq (C+1) \frac{\mathbf{o}(\delta+h_n(\xi+\delta,\eta)-h_n(\xi,\eta))}{|\delta+h_n(\xi+\delta,\eta)-h_n(\xi,\eta)|}.
\end{split}
\end{equation}
Consequently, since $\xi \to h_n(\xi,\eta)$ is continuous, we have that
\begin{displaymath}
\lim_{\delta\to 0}|\delta+h_n(\xi+\delta,\eta)-h_n(\xi,\eta)|=0,
\end{displaymath}
and thus
\begin{displaymath}
\lim_{\delta\to 0}\frac{\mathbf{o}(\delta+h_n(\xi+\delta,\eta)-h_n(\xi,\eta))}{|\delta+h_n(\xi+\delta,\eta)-h_n(\xi,\eta)|}=0.
\end{displaymath}
This fact combined with \eqref{eq: final eq est o hn} implies that 
\begin{displaymath}
\mathbf{o}(\delta+h_n(\xi+\delta,\eta)-h_n(\xi,\eta))=\mathbf{o}(\delta)
\end{displaymath}
which together with~\eqref{eq: est hn o} yields that 
\begin{equation*}
\begin{split}
h_n(\xi+\delta,\eta)-h_n(\xi,\eta) =R\delta+\mathbf{o}(\delta).\\
\end{split}
\end{equation*}
Hence, $\xi \to h_n(\xi,\eta)$ is differentiable and  $\frac{\partial h_n(\xi, \eta) }{\partial \xi}=R$. Moreover, it follows from the expression of $R$ given in \eqref{eq: def of R} and the fact that $\xi \to \bar h_n(\xi,\eta)$ is $C^1$ that $\xi \to h_n(\xi,\eta)$ is also $C^1$. The proof of the theorem is completed.
\end{proof}

\begin{example}\label{remm}
It is easy to construct  examples where property \ref{H: Kn+Jn finite} is satisfied for every $n\in \Z$ simultaneously. For instance, suppose that
\[
M:=\sup_{k\in \Z} \max \{ |A_k|, |A_k^{-1}|, |P_k|\}<+\infty.
\]
Without any loss of generality, we may assume that $M\ge 1$.
For $k\in \Z$, take $\gamma_k$ such that $0\le \gamma_k \le \frac{1}{(2M)^{2|k|+1}}$.
 Observe that $\frac{|A_j^{-1}|}{1-\gamma_j|A_j^{-1}|}\le 2M$ for $k\in \Z$. Consequently,
\begin{displaymath}
\begin{split}
K_n&=\sum _{k< n} |\mathcal{G} (n,k+1)|\gamma_k \left(\prod_{j=k}^{n-1} \frac{|A_j^{-1}|}{1-\gamma_j|A_j^{-1}|}\right)\\
&\leq \sum _{k< n} M^{n-k} \frac{1}{(2M)^{2|k|}} (2M)^{n-k}\\
&\leq (2M)^{2|n|}\sum _{k< n} \frac{1}{2^{|k|}} <+\infty,
\end{split}
\end{displaymath}
for every $n\in \Z$. Similarly, since $|A_k|+\gamma_k\le 2M$, it follows that
\begin{displaymath}
\begin{split}
J_n &=\sum _{k> n} |\mathcal{G} (n,k+1)|\gamma_k \left(\prod_{j=n}^{k-1} (|A_j|+\gamma_j)\right) \\
&\le \sum_{k>n} M^{k+2-n}\frac{1}{(2M)^{2|k|}} (2M)^{k-n} \\
&\le M^2(2M)^{2|n|}\sum _{k> n} \frac{1}{2^{|k|}} <+\infty,
\end{split}
\end{displaymath}
for $n\in \Z$.

We conclude that Theorem~\ref{theo: reg H bar} is applicable under mild assumptions on the linear part in~\eqref{lnd}, which in particular  does not have to exhibit any asymptotic behaviour. For example, we can take $A_k=P_k=\Id$ for each $k\in \Z$.

As for \ref{H: Kn+Jn smaller 1}, fix $n\in \Z$ and choose $n_0\in \N$ such that $M^2(2M)^{2|n|}\sum _{k\in\Z} \frac{1}{2^{|k|+n_0}} <1$. Then, if $\gamma_k<\frac{1}{(2M)^{2|k|+n_0}}$ for every $k\in \Z$, our assumptions combined with the previous estimates implies that \ref{H: Kn+Jn smaller 1} is satisfied for $n$.
\end{example}

We now present examples under which~\ref{H: Kn+Jn smaller 1} holds for each $n\in \Z$.

\begin{example}\label{EX1}
We consider the case when $X$ is a product of a Banach space $\mathcal X=(\mathcal X, | \cdot |_{\mathcal X})$ with itself, i.e. $X=\mathcal X\times \mathcal X$. Then, $X$ is a Banach space with respect to the norm $|(x_1, x_2)|:=\max \{ |x_1|_{\mathcal X}, |x_2|_{\mathcal X}\}$, $x_1,x_2\in \mathcal X$. 

Take an arbitrary $\lambda >0$. Let us consider sequences of bounded linear operators $(A_n)_{n\in \Z}$ and $(P_n)_{n\in \Z}$ acting on $X$ given by
\begin{displaymath}
A_n=\begin{pmatrix}
e^\lambda \Id  & 0 \\
0 & e^{-\lambda} \Id \\
\end{pmatrix} \quad \text{and} \quad
P_n=\begin{pmatrix}
0 & 0 \\
0 & \Id \\
\end{pmatrix},
\end{displaymath}
for $n\in \Z$, where $\Id$ denotes the identity operator on $\mathcal X$.

Then,
\[
\mathcal G(m, n)=\begin{cases}
\begin{pmatrix} 
0 & 0 \\
0 & e^{-\lambda (m-n)}\Id
\end{pmatrix}  & \text{for $m\ge n$;} \\
-\begin{pmatrix}
e^{-\lambda (n-m)}\Id & 0 \\
0 &0\\
\end{pmatrix} &\text{for $m<n$.}
\end{cases}
\]
Consequently,
\[
|\mathcal G(m, n)|=e^{-\lambda |m-n|}.
\]
Set
\[
M:=\prod_{j\in \Z}\left (1+e^{-\lambda |j|} \right )<+\infty.
\]
Now, let us consider a sequence $(\gamma_k)_{k\in \Z}$ of nonnegative numbers such that:
\begin{equation}\label{1154}
\gamma_k \le \frac{1}{e^{\lambda (|k|+1)}+e^{\lambda}} \quad \text{for  $k\in \Z$},
\end{equation}
and
\begin{equation}\label{1155}
\sum_{k\in \Z}\gamma_k <\frac{1}{e^\lambda M}.
\end{equation}
Observe that it follows from~\eqref{1154} that
\begin{equation}\label{1156}
\frac{1}{1-\gamma_k}\le \frac{1}{1-e^\lambda \gamma_k} \le 1+e^{-\lambda |k|}, \quad k\in \Z.
\end{equation}

Then,  using~\eqref{1156}  we have that 
\begin{displaymath}
\begin{split}
K_n&=\sum _{k< n} |\mathcal{G} (n,k+1)|\gamma_k \left(\prod_{j=k}^{n-1} \frac{|A_j^{-1}|}{1-\gamma_j|A_j^{-1}|}\right)\\
&\leq \sum _{k< n} e^{-\lambda (n-k-1)} \gamma_k \left(\prod_{j=k}^{n-1} \frac{e^\lambda}{1-e^\lambda \gamma_j}\right) \\
&\leq \sum _{k< n} e^{-\lambda (n-k-1)} \gamma_k e^{\lambda (n-k)} \left(\prod_{j=k}^{n-1} \frac{1}{1-e^\lambda \gamma_j}\right) \\
&\leq e^\lambda \sum _{k< n} \gamma_k \left(\prod_{j=k}^{n-1} \frac{1}{1-e^\lambda \gamma_j}\right) \\
&\le e^\lambda \sum _{k< n} \gamma_k \left(\prod_{j=k}^{n-1} \left (1+e^{-\lambda |j|} \right ) \right)  \\
&\leq e^\lambda M \sum _{k< n} \gamma_k. \\
\end{split}
\end{displaymath}

Similarly, since $1+e^{-\lambda}\gamma_j\leq 1+\gamma_j\leq \frac{1}{1-\gamma_j}$, it follows from~\eqref{1156} that 
\[
\begin{split}
J_n &=\sum _{k> n} |\mathcal{G} (n,k+1)|\gamma_k \left(\prod_{j=n}^{k-1} (|A_j|+\gamma_j)\right) \\
&\leq \sum _{k> n} e^{-\lambda (k+1-n)} \gamma_k \left(\prod_{j=n}^{k-1} (e^\lambda +\gamma_j)\right) \\
&= \sum _{k> n} e^{-\lambda (k+1-n)} \gamma_k e^{\lambda (k-n)} \left(\prod_{j=n}^{k-1} (1+e^{-\lambda} \gamma_j)\right) \\
&\leq e^{-\lambda} M\sum _{k> n} \gamma_k \\
&\le e^{\lambda} M\sum _{k> n} \gamma_k.
\end{split}
\]

Finally, since $M\ge 1$ we have that  $|\mathcal{G} (n,n+1)|\gamma_n=e^{-\lambda} \gamma_n\le e^\lambda M\gamma_n$. Thus, from~\eqref{1155} we conclude that 
\[K_n+J_n+|\mathcal{G} (n,n+1)|\gamma_n\le e^\lambda M \sum_{n\in \Z}\gamma_n <1,\]
for each $n\in \Z$.
\end{example}

\begin{example}\label{EX2} 
Let $X$ be as in Example~\ref{EX1}. Furthermore, take a sequence of isometries
 $(B_n)_{n\in \Z}$   acting $\mathcal X$ and let $(\theta_n)_{n\in \Z}$ be any sequence of numbers satisfying $\theta_n \ge 1$ for every $n\in \Z$. Moreover, suppose that $\theta_n\leq \theta_{n+1}$ and that there exists a constant $T\geq 1$ such that $\frac{\theta_{n+1}}{\theta_n}\leq T$ for every $n\in\Z$. We now consider sequences  of bounded linear operators $(A_n)_{n\in \Z}$ and $(P_n)_{n\in \Z}$ acting on $X$ given by
\begin{displaymath}
A_n=\begin{pmatrix}
\frac{\theta_n}{\theta_{n+1}}\Id  & 0 \\
0 & B_n \\
\end{pmatrix} \quad \text{and} \quad
P_n=\begin{pmatrix}
\Id & 0 \\
0 & 0 \\
\end{pmatrix},
\end{displaymath}
for $n\in \Z$.
Then,
\[
\mathcal G(m, n)=\begin{cases}
\begin{pmatrix} 
\frac{\theta_n}{\theta_{m}}\Id &  0 \\
0 & 0  \\
\end{pmatrix}  & \text{for $m\ge n$;} \\
-\begin{pmatrix}
0 & 0 \\
0 &\mathcal{B}(m,n)\\
\end{pmatrix} &\text{for $m<n$,}
\end{cases}
\]
where 
\[
\mathcal B (m, n)=\begin{cases}
B_{m-1}\cdots B_n & \text{for $m>n$;}\\
\Id  &\text{for $m=n$;} \\
B_m^{-1}\ldots B_{n-1}^{-1}& \text{for $m<n$.}\\
\end{cases}
\]
Consequently,
\[
|\mathcal G(m, n)|=\begin{cases}
\frac{\theta_n}{\theta_{m}}  & \text{for $m\ge n$;} \\
1 &\text{for $m<n$.}
\end{cases}
\]
Set
\[
M:=\prod_{j\in \Z}\left(1+\frac{1}{2^{|j|}}\right)<+\infty.
\]

Now, let us consider a sequence $(\gamma_k)_k$ of nonnegative numbers such that:
\begin{equation}\label{2055}
\gamma_k<\frac{1}{T(2^{|k|+1}+1)} \quad \text{for $k\in \Z$,}
\end{equation}
and
\begin{equation}\label{2056}
\sum_{k\in \Z}\gamma_k <\frac{1}{TM}.
\end{equation}
Observe that it follows from~\eqref{2055} that 
\begin{equation}\label{2059}
\frac{1}{1-\gamma_k}\leq \frac{1}{1-T\gamma_k}\leq 1+\frac{1}{2^{|k|}}, \quad k\in \Z.
\end{equation}

Then, observing that $|A^{-1}_j|=\frac{\theta_{j+1}}{\theta_{j}}$ and using~\eqref{2059}, we get that
\begin{displaymath}
\begin{split}
K_n&=\sum _{k< n} |\mathcal{G} (n,k+1)|\gamma_k \left(\prod_{j=k}^{n-1} \frac{|A_j^{-1}|}{1-\gamma_j|A_j^{-1}|}\right)\\
&= \sum _{k< n} \frac{\theta_{k+1} }{\theta_n}\gamma_k \left(\prod_{j=k}^{n-1} \frac{\left(\frac{\theta_{j+1}}{\theta_{j}}\right)}{1-\left(\frac{\theta_{j+1}}{\theta_{j}}\right)\gamma_j}\right) \\
&\leq \sum _{k< n} \frac{\theta_{k+1} }{\theta_n}\gamma_k \left(\prod_{j=k}^{n-1} \frac{\theta_{j+1}}{\theta_{j}}\left(\frac{1}{1-T\gamma_j}\right)\right) \\
&= \sum _{k< n} \frac{\theta_{k+1} }{\theta_n}\gamma_k \frac{\theta_n}{\theta_k} \left(\prod_{j=k}^{n-1} \frac{1}{1-T\gamma_j}\right) \\
&\leq \sum _{k< n} TM\gamma_k.  \\
\end{split}
\end{displaymath}

Similarly, observing that $|A_j|=1$ and $1+\gamma_j\leq \frac{1}{1-\gamma_j}\leq \frac{1}{1-T\gamma_j}$, we have that 
\[
\begin{split}
J_n &=\sum _{k> n} |\mathcal{G} (n,k+1)|\gamma_k \left(\prod_{j=n}^{k-1} (|A_j|+\gamma_j)\right) \\
&\leq \sum _{k> n} \gamma_k \left(\prod_{j=n}^{k-1} (1+\gamma_j)\right) \\
&\leq M\sum _{k> n} \gamma_k \\
&\leq TM\sum _{k> n} \gamma_k .
\end{split}
\]
Finally, since $M, T\geq 1$, it follows that $|\mathcal{G} (n,n+1)|\gamma_n=\gamma_n\le TM\gamma_n$. Thus, from~\eqref{2056} we conclude that 
\[K_n+J_n+|\mathcal{G} (n,n+1)|\gamma_n \le TM\sum_{k\in \Z} \gamma_k<1,\]
 for every $n\in \Z$. 
\end{example}

\subsection{An important particular case}

We will now restrict to the case when $Y$ is a trivial Banach space, i.e. $Y=\{0\}$. In this case, \eqref{nnd} reduces to
\begin{equation}\label{nndd}
x_{n+1}=A_n x_n+f_n(x_n),
\end{equation}
where $f_n\colon X\to X$, while~\eqref{lnd} is equivalent to
\begin{equation}\label{lndd}
x_{n+1}=A_n x_n.
\end{equation}

The following result (originally established in~\cite{RS}) is a consequence of Theorem~\ref{theo: BDP disc}.
\begin{corollary}
Suppose that \ref{H: fn is Lipschitz and bounded}, \ref{H: N finite} and \ref{H: q smaller 1} are satisfied. Then, there exists a sequence of continuous functions $H_n\colon X\to X$, $n\in \Z$ of the form $H_n(x)=x+h_n(x)$, where  $\sup_{n,x}|h_n(x)|<\infty$, such that if $n\mapsto x_n$ is a solution of (\ref{lndd}),
then $n\mapsto H_n(x_n)$ is a solution of (\ref{nndd}). In addition, there exists a sequence of continuous functions $\bar H_n \colon X \to X$, $n\in \Z$ of the form $\bar H_n( x)=x+\bar h_n(x)$, where $\sup_{n,x}|\bar h_n(x)|<\infty$,  such that if $n\mapsto x_n$ is a solution of (\ref{nndd}),
then $n\mapsto \bar H_n(x_n)$ is a solution of (\ref{lndd}). Moreover, $H_n$ and $\bar H_n$ are inverses of
each other.

\end{corollary}

The following result is a direct consequence of Theorems~\ref{theo: reg H bar} and~\ref{theo: reg H}.
\begin{corollary}\label{newcor}
Suppose hypothesis \ref{H: fn is Lipschitz and bounded}, \ref{H: N finite}, \ref{H: q smaller 1}, \ref{H: Aj gamma j smaller 1} and \ref{H: fn diff first var} are satisfied. Moreover,  assume that the  hypothesis \ref{H: Kn+Jn smaller 1} is satisfied for each $n\in \Z$. Then, 
$H_n$ and $\bar H_n$ are of class  $C^1$ for every $n\in \Z$.
\end{corollary}

\begin{remark}\label{EMO}
We would like to compare Corollary~\ref{newcor} with the smooth linearization result established in~\cite[Theorem 2]{DZZ}. We stress that~\cite[Theorem 2]{DZZ} requires that~\eqref{lndd} admits a (nonuniform) exponential dichotomy. In particular, it is applicable 
when the sequence $(A_n)_{n\in \Z}$ is as in Example~\ref{EX1} and when $(\gamma_k)_{k\in \Z}$ is a constant sequence, i.e. $\gamma_k=c$, $k\in \Z$, where $c$ is sufficiently small. We note that this situation is not covered with Corollary~\ref{newcor}. Indeed, for any $n\in \Z$
 we have that
\[
\begin{split}
J_n &=\sum _{k> n} |\mathcal{G} (n,k+1)|\gamma_k \left(\prod_{j=n}^{k-1} (|A_j|+\gamma_j)\right) \\
&=c \sum _{k> n} e^{-\lambda (k+1-n)}  \left(\prod_{j=n}^{k-1} (e^\lambda +c)\right) \\
&\ge c\sum _{k> n} e^{-\lambda (k+1-n)} e^{\lambda (k-n)} \\
&=+\infty,
\end{split}
\]
and thus~\ref{H: Kn+Jn smaller 1} cannot hold.

On the other hand, as illustrated by Example~\ref{EX2}, Corollary~\ref{newcor} is applicable in situations when~\eqref{lndd} does not admit an exponential dichotomy. We conclude that Corollary~\ref{newcor} and~\cite[Theorem 2]{DZZ} complement each other.
\end{remark}

\begin{remark}
We note that a result similar to Corollary~\ref{newcor} has been established in~\cite[Theorem 3.2]{CJ} under some additional assumptions on~\eqref{lndd} and assuming that $X$ is finite-dimensional. 
\end{remark}

\subsection{Regularity with respect to the second variable}\label{sec: reg second var}
We are now interested in formulating sufficient conditions under which $H_n$ and $\bar H_n$ are smooth in the second variable.

We start studying the regularity properties of the maps $\eta \to y(k,n,\eta)$ and $\eta \to x_2(k,n,\xi,\eta)$.
Consider
\begin{equation}\label{Dkn}
 D_{k,n}=\begin{cases}
\prod _{j=n}^{k-1}\tau_j & \text{for $k>n$;}\\
1  &\text{for $k=n$;} \\
\prod_{j=k}^{n-1} \sigma_j & \text{for $k<n$,}
\end{cases}
\end{equation}
and
\begin{equation}\label{Mkn}
M_{k,n}=\begin{cases}
 \prod _{j=n}^{k-1}(|A_j|+\gamma_j+\max\{\rho _{j},\tau _{j}\}) & \text{for $k>n$;}\\
1  &\text{for $k=n$;} \\
\prod_{j=k}^{n-1} \left(\frac{|A_{j}^{-1}|}{1-\gamma_{j}|A_{j}^{-1}|}+\sigma_{j}\right) & \text{for $k<n$.}
\end{cases}
\end{equation}

\begin{lemma} \label{lemma: y sol is lips}
Suppose \ref{H: gn is Lipschitz} is satisfied. Then, for any $\eta, \tilde{\eta}\in Y$ and $n,k\in\Z$,
\begin{displaymath}
|y(k,n,\eta)-y(k,n, \tilde \eta)|\leq D_{k,n}|\eta - \tilde{\eta}|.
\end{displaymath}
\end{lemma}
\begin{proof}
This follows easily from the definition of $y(k,n,\eta)$ and hypothesis \ref{H: gn is Lipschitz}.
\end{proof}

\begin{lemma} \label{lemma: sol x2 lips second var}
Suppose \ref{H: fn is Lipschitz and bounded}, \ref{H: Aj gamma j smaller 1}, \ref{H: fn is Lipschitz second var}, \ref{H: gn is Lipschitz} and \ref{H: sigma rho smaller than 1} are satisfied. Then, for any $\xi\in X$, $\eta,\tilde{\eta} \in Y$ and $n,k\in\Z$, we have that 
\begin{displaymath}
|x_2(k,n,\xi,\eta)-x_2(k,n,\xi,\tilde{\eta})|\leq M_{k,n}|\eta - \tilde{\eta}|.
\end{displaymath}
\end{lemma}
\begin{proof}
Fix $n\in \Z$. By \eqref{nnd} we have that
\begin{equation}\label{eq: x2 future second var}
\begin{split}
x_2(n+1,n,\xi,\eta)&=A_n\xi+f_n(\xi,\eta).\\
\end{split}
\end{equation}
Consequently, using \ref{H: fn is Lipschitz second var} we have that 
\begin{displaymath}
\begin{split}
|x_2(n+1,n,\xi,\eta)-x_2(n+1,n,\xi,\tilde{\eta})|&=|f_n(\xi,\eta)-f_n(\xi,\tilde{\eta})|\\
&\leq \rho _n|\eta-\tilde{\eta}|.
\end{split}
\end{displaymath}
Moreover, using Lemmas \ref{lemma: sol lips} and~\ref{lemma: y sol is lips} and the previous observation we get that
\begin{displaymath}
\begin{split}
& |x_2(n+2,n,\xi,\eta)-x_2(n+2,n,\xi,\tilde{\eta})|\\
&=|x_2(n+2,n+1,x_2(n+1,n,\xi,\eta),y(n+1,n,\eta))\\
&\phantom{=}-x_2(n+2,n+1,x_2(n+1,n,\xi,\tilde{\eta}),y(n+1,n,\tilde{\eta}))|\\
&\leq |x_2(n+2,n+1,x_2(n+1,n,\xi,\eta),y(n+1,n,\eta))\\
&\phantom{\leq}-x_2(n+2,n+1,x_2(n+1,n,\xi,\tilde{\eta}),y(n+1,n,\eta))|\\
&\phantom{=}+|x_2(n+2,n+1,x_2(n+1,n,\xi,\tilde{\eta}),y(n+1,n,\eta)) \\
&\phantom{=}-x_2(n+2,n+1,x_2(n+1,n,\xi,\tilde{\eta}),y(n+1,n,\tilde{\eta}))|\\
&\leq C_{n+2,n+1}|x_2(n+1,n,\xi,\eta)-x_2(n+1,n,\xi,\tilde{\eta})|\\
&\phantom{\leq}+\rho _{n+1}|y(n+1,n,\eta)-y(n+1,n,\tilde{\eta})|\\
&\leq C_{n+2,n+1}\rho_n |\eta-\tilde{\eta}| +\rho _{n+1}\tau_n|\eta-\tilde{\eta}|\\
&\leq \left(|A_{n+1}| +\gamma_{n+1} +\max\{\rho _{n+1},\tau _{n+1}\}\right)\max\{\rho_n,\tau_n\}|\eta-\tilde{\eta}|\\
&\leq \left(|A_{n+1}| +\gamma_{n+1} +\max\{\rho _{n+1},\tau _{n+1}\}\right)\left(|A_{n}| +\gamma_{n} +\max\{\rho _{n},\tau _{n}\}\right)|\eta-\tilde{\eta}|.
\end{split}
\end{displaymath}
Using that
\begin{displaymath}
x_2(k,n,\xi,\eta)=x_2(k,m,x_2(m,n,\xi,\eta),y(m,n,\eta)),
\end{displaymath} 
and proceeding inductively we conclude that 
\begin{displaymath}
\begin{split}
|x_2(k,n,\xi,\eta)-x_2(k,n,\xi,\tilde{\eta})|&\leq \prod _{j=n}^{k-1}(|A_j|+\gamma_j+\max\{\rho _{j},\tau _{j}\})|\eta - \tilde{\eta}|\\
&=M_{k,n}|\eta - \tilde{\eta}|,
\end{split}
\end{displaymath}
for every $k> n$.

We now consider  the case when $k<n$. Let $T_j:X\times Y\to X$ be given by~\eqref{Tj}. Given $\xi\in X$ and $\eta,\tilde \eta \in Y$, using \ref{H: fn is Lipschitz and bounded} we get that
\begin{equation}\label{eq: estimative Tj cont second var}
\begin{split}
|T_j(\xi,\eta)-T_j(\xi,\tilde \eta)|&=\left|A_j^{-1}\left( f_j(T_j(\xi,\tilde \eta),\tilde \eta)-f_j(T_j(\xi, \eta), \eta)\right)\right|\\
&\leq |A_j^{-1}|\left| f_j(T_j(\xi,\tilde \eta),\tilde \eta)-f_j(T_j(\xi, \eta), \tilde \eta)\right|\\
&\phantom{=}+|A_j^{-1}|\left| f_j(T_j(\xi,\eta),\tilde \eta)-f_j(T_j(\xi, \eta),  \eta)\right|\\
&\leq |A_j^{-1}|\gamma_j\left| T_j(\xi,\tilde \eta)-T_j(\xi, \eta)\right|\\
&\phantom{=}+|A_j^{-1}|\left| f_j(T_j(\xi,\eta),\tilde \eta)-f_j(T_j(\xi, \eta),  \eta)\right|,\\
\end{split}
\end{equation}
which implies that
\begin{displaymath}
\begin{split}
\left(1-\gamma_j|A_j^{-1}|\right)|T_j(\xi,\eta)-T_j(\xi,\tilde \eta)|&\leq |A_j^{-1}|\left| f_j(T_j(\xi,\eta),\tilde \eta)-f_j(T_j(\xi, \eta),  \eta)\right|\\
&\leq \rho_j|A_j^{-1}|\left| \eta- \tilde  \eta\right|.
\end{split}
\end{displaymath}
Thus, combining this observation with \ref{H: Aj gamma j smaller 1} we conclude that
\begin{displaymath}
\begin{split}
|T_j(\xi,\eta)-T_j(\xi,\tilde \eta)|&\leq \frac{\rho_j|A_j^{-1}|}{1-\gamma_j|A_j^{-1}|}\left| \eta- \tilde  \eta\right|.
\end{split}
\end{displaymath}
Consequently, by \eqref{eq: Tn x2}, Lemma \ref{lemma: y sol is lips} and the the previous observation, we get that
\begin{displaymath}
\begin{split}
&|x_2(n-1,n,\xi,\eta)-x_2(n-1,n,\xi,\tilde \eta)| \\
&= |T_{n-1}(\xi,y(n-1,n,\eta))-T_{n-1}(\xi,y(n-1,n,\tilde \eta))|\\
&\leq \frac{\rho_{n-1}|A_{n-1}^{-1}|}{1-\gamma_{n-1}|A_{n-1}^{-1}|}\left| y(n-1,n, \eta)- y(n-1,n,\tilde \eta)\right|\\
&\leq  \frac{\sigma_{n-1}\rho_{n-1}|A_{n-1}^{-1}|}{1-\gamma_{n-1}|A_{n-1}^{-1}|}\left| \eta- \tilde \eta\right|.\\
\end{split}
\end{displaymath}

Moreover, using Lemmas \ref{lemma: sol lips} and~\ref{lemma: y sol is lips}, the previous observation and \ref{H: sigma rho smaller than 1}, we obtain that
\begin{displaymath}
\begin{split}
& |x_2(n-2,n,\xi,\eta)-x_2(n-2,n,\xi,\tilde{\eta})|\\
&=|x_2(n-2,n-1,x_2(n-1,n,\xi,\eta),y(n-1,n,\eta))\\
&\phantom{=}-x_2(n-2,n-1,x_2(n-1,n,\xi,\tilde{\eta}),y(n-1,n,\tilde{\eta}))|\\
&\leq |x_2(n-2,n-1,x_2(n-1,n,\xi,\eta),y(n-1,n,\eta))\\
&\phantom{\leq}-x_2(n-2,n-1,x_2(n-1,n,\xi,\tilde{\eta}),y(n-1,n,\eta))|\\
&\phantom{=}+|x_2(n-2,n-1,x_2(n-1,n,\xi,\tilde{\eta}),y(n-1,n,\eta))\\
&\phantom{=}-x_2(n-2,n-1,x_2(n-1,n,\xi,\tilde{\eta}),y(n-1,n,\tilde{\eta}))|\\
&\leq C_{n-2,n-1}|x_2(n-1,n,\xi,\eta)-x_2(n-1,n,\xi,\tilde{\eta})|\\
&\phantom{\leq}+ \frac{\sigma_{n-2}\rho_{n-2}|A_{n-2}^{-1}|}{1-\gamma_{n-2}|A_{n-2}^{-1}|}|y(n-1,n,\eta)-y(n-1,n,\tilde{\eta})|\\
&\leq \frac{|A_{n-2}^{-1}|}{1-\gamma_{n-2}|A_{n-2}^{-1}|}\frac{\sigma_{n-1}\rho_{n-1}|A_{n-1}^{-1}|}{1-\gamma_{n-1}|A_{n-1}^{-1}|}|\eta-\tilde{\eta}|+ \frac{\sigma_{n-2}\rho_{n-2}|A_{n-2}^{-1}|}{1-\gamma_{n-2}|A_{n-2}^{-1}|}\sigma_{n-1}|\eta-\tilde{\eta}|\\
&\leq \frac{|A_{n-2}^{-1}|}{1-\gamma_{n-2}|A_{n-2}^{-1}|}\frac{|A_{n-1}^{-1}|}{1-\gamma_{n-1}|A_{n-1}^{-1}|}|\eta-\tilde{\eta}|+ \frac{|A_{n-2}^{-1}|}{1-\gamma_{n-2}|A_{n-2}^{-1}|}\sigma_{n-1}|\eta-\tilde{\eta}|\\
&\leq \left(\frac{|A_{n-2}^{-1}|}{1-\gamma_{n-2}|A_{n-2}^{-1}|}+\sigma_{n-2}\right)\left(\frac{|A_{n-1}^{-1}|}{1-\gamma_{n-1}|A_{n-1}^{-1}|}+\sigma_{n-1}\right)|\eta-\tilde{\eta}|.
\end{split}
\end{displaymath}
Using that 
\begin{displaymath}
x_2(k,n,\xi,\eta)=x_2(k,m,x_2(m,n,\xi,\eta),y(m,n,\eta)),
\end{displaymath} 
and proceeding recursively we conclude that 
\begin{displaymath}
\begin{split}
|x_2(k,n,\xi,\eta)-x_2(k,n,\xi,\tilde{\eta})|&\leq \prod _{j=k}^{n-1}\left(\frac{|A_{j}^{-1}|}{1-\gamma_{j}|A_{j}^{-1}|}+\sigma_{j}\right)|\eta - \tilde{\eta}|\\
&=M_{k,n}|\eta - \tilde{\eta}|,
\end{split}
\end{displaymath}
for every $k<n$.

Finally, since $x_2(n,n,\xi,\eta)=\xi$ the proof is complete.
\end{proof}

\begin{lemma}\label{lemma: diff of y}
Suppose \ref{H: gn diff} is satisfied. Then,  the map $\eta \to y(k,n,\eta)$ is of class $C^r$, $r\geq 1$, for every $n,k\in\mathbb{Z}$.
\end{lemma}
\begin{proof}
This is a direct consequence of the definition and hypothesis \ref{H: gn diff}.
\end{proof}

\begin{proposition}\label{prop: diff of sol second var}
Suppose \ref{H: fn is Lipschitz and bounded}, \ref{H: Aj gamma j smaller 1}, \ref{H: fn diff first var}, \ref{H: fn diff second var} and \ref{H: gn diff} are satisfied. Then the map $\eta \to x_2(k,n,\xi,\eta)$ is of class $C^r$, $r\geq 1$, for every  $\xi \in X$ and $n,k\in\mathbb{Z}$.
\end{proposition}
\begin{proof}
Fix $n\in \Z$. By \eqref{nnd} we have that
\begin{displaymath}
\begin{split}
x_2(n+1,n,\xi,\eta)&=A_n\xi+f_n(\xi,\eta).\\
\end{split}
\end{displaymath}
Consequently, by \ref{H: fn diff second var} it follows that $\eta \to x_2(n+1,n,\xi,\eta)$ is of class $C^r$. Thus, using that
\begin{displaymath}
x_2(k,n,\xi,\eta)=x_2(k,m,x_2(m,n,\xi,\eta),y(m,n,\eta)),
\end{displaymath}
Proposition \ref{prop: diff of sol}, Lemma \ref{lemma: diff of y}, hypothesis \ref{H: fn diff first var} and \ref{H: fn diff second var} and proceeding inductively, it follows that $\eta \to x_2(k,n,\xi,\eta)$ is of class $C^r$ for every $k>n$.

We now consider the case when $k<n$. Let $T_j:X\times Y\to X$ be given by~\eqref{Tj}.
As in the proof of Proposition \ref{prop: diff of sol}, the result will follow once we prove that  $\eta \to T_j(\xi,\eta)$ is of class $C^r$. We start observing that $\eta \to T_j(\xi,\eta)$ is continuous. Indeed, by \eqref{eq: estimative Tj cont second var} we get that
\begin{displaymath}
\begin{split}
\left(1-\gamma_j|A_j^{-1}|\right)|T_j(\xi,\eta)-T_j(\xi,\tilde \eta)|&\leq |A_j^{-1}|\left| f_j(T_j(\xi,\eta),\tilde \eta)-f_j(T_j(\xi, \eta),  \eta)\right|.\\
\end{split}
\end{displaymath}
Consequently, by \ref{H: Aj gamma j smaller 1} and \ref{H: fn diff second var}, it follows that $|T_j(\xi,\eta)-T_j(\xi,\tilde \eta)|\to 0$ whenever $\tilde{\eta}\to \eta$. We conclude that $\eta \to T_j(\xi,\eta)$ is continuous. 

Let us now consider
\begin{equation*}\label{eq: def of L tilde}
\tilde{L}: =-\left(\Id+A_j^{-1}\frac{\partial f_j(T_j(\xi,\eta),\eta))}{\partial u}\right)^{-1}A_j^{-1} \frac{\partial f_j(T_j(\xi,\eta),\eta))}{\partial v},
\end{equation*}
where $\frac{\partial f_j(\cdot,\cdot) }{\partial u}$ and $\frac{\partial f_j(\cdot,\cdot) }{\partial v}$ denote the derivatives of $f_j$ with respect to the first and second variable, respectively. Observe that, since $|A_j^{-1}\frac{\partial f_j(T_j(\xi,\eta),\eta))}{\partial u}|<1$, the operator $\tilde{L}$ is well defined. Now, since $\eta \to T_j(\xi,\eta)$ is continuous, one may proceed similarly to  the proof of Theorem \ref{theo: reg H}, and obtain that $\eta \to T_j(\xi,\eta)$ is differentiable and $\frac{\partial T_j(\xi,\eta) }{\partial \eta}=\tilde{L}$. Thus, from our hypothesis and the definition of $\tilde{L}$ it follows that $\eta \to T_j(\xi,\eta)$ is in fact of class $C^r$. Consequently, combining this observation with \eqref{eq: Tn x2} and Lemma \ref{lemma: diff of y}, it follows that $\eta \to x_2(k,n,\xi,\eta)$ is of class $C^r$ for every $k<n$. Finally, since $x_2(n,n,\xi,\eta)=\xi$ for any $\eta$ the desired conclusion holds in the case $k=n$ also.
\end{proof}

\begin{theorem}\label{theo: reg H bar second var} 
Suppose hypothesis \ref{H: fn is Lipschitz and bounded}, \ref{H: N finite}, \ref{H: q smaller 1}, \ref{H: Aj gamma j smaller 1}, \ref{H: fn diff first var}, \ref{H: fn is Lipschitz second var}, \ref{H: gn is Lipschitz}, \ref{H: sigma rho smaller than 1}, \ref{H: fn diff second var} and  \ref{H: gn diff} are satisfied. Moreover, given $n\in \Z$ assume that hypothesis \ref{H: temporary name} is satisfied for $n$. Then, the map $\eta \to \bar{H}_n(\xi,\eta)$ is $C^1$.
\end{theorem}

\begin{proof} We proceed as in the proof of Theorem \ref{theo: reg H bar}. Fix $n\in \Z$ and recall that 
\[ \bar H_n(\xi,\eta)=(\xi+\bar h_n(\xi,\eta),\eta)\]
where
\begin{equation}\label{eq: hn bar 2}
\bar h_n(\xi, \eta)=-\sum_{k\in \Z} \mathcal G(n, k+1)f_{k}(x_2(k, n, \xi, \eta), y(k, n, \eta)),
\end{equation}
for $(\xi, \eta)\in X\times Y$ (see the proof of \cite[Theorem 3.1]{BDP}). Thus, it remains to prove  that the map $\eta \to \bar h_n(\xi,\eta)$ is $C^1$.

By hypothesis \ref{H: fn diff first var} and \ref{H: fn diff second var}, Lemma \ref{lemma: diff of y} and Proposition \ref{prop: diff of sol second var} we have that each of the terms involved in the series \eqref{eq: hn bar 2} are differentiable with respect to $\eta$. Moreover,
\begin{equation}\label{2113}
\begin{split}
\frac{\partial f_{k}(x_2(k, n, \xi, \eta), y(k, n, \eta))}{\partial \eta}&= \frac{\partial f_{k}(x_2(k, n, \xi, \eta), y(k, n, \eta))}{\partial u}\frac{\partial x_2(k, n, \xi, \eta)}{\partial \eta} \\
&\phantom{=} + \frac{\partial f_{k}(x_2(k, n, \xi, \eta), y(k, n, \eta))}{\partial v}\frac{\partial y(k, n, \eta)}{\partial \eta},
\end{split}
\end{equation}
where $\frac{\partial f_k(\cdot,\cdot) }{\partial u}$ and $\frac{\partial f_k(\cdot,\cdot) }{\partial v}$ denote the derivative of $f_k$ with respect to the first and second variables, respectively. Furthermore, by \eqref{eq: bound dif of f} we have that
\begin{displaymath}
\begin{split}
\left|\frac{\partial f_k(u,v)}{\partial u}\right|&\leq \gamma_k.
\end{split}
\end{displaymath}
Similarly, using respectively \ref{H: fn is Lipschitz second var}, Lemma \ref{lemma: y sol is lips} and Lemma \ref{lemma: sol x2 lips second var}, we get that 
\begin{displaymath}
\begin{split}
\left|\frac{\partial f_k(u,v)}{\partial v}\right|\leq \rho_k,\quad  \left|\frac{\partial  y(k, n, \eta)}{\partial \eta}\right|\leq D_{k,n} \text{ and } \left|\frac{\partial x_2(k, n, \xi, \eta)}{\partial \eta }\right| \leq M_{k,n}.
\end{split}
\end{displaymath}
Combining these facts with~\eqref{2113},  we conclude that
\begin{equation}\label{2116}
\begin{split}
\left|\frac{\partial f_{k}(x_2(k, n, \xi, \eta), y(k, n, \eta))}{\partial \eta} \right|\leq \gamma_k M_{k,n}+\rho_k D_{k,n}.
\end{split}
\end{equation}

By~\ref{H: temporary name} and~\eqref{2116}, we have  that
\begin{displaymath}
\begin{split}
&\sum_{k\in \Z} \left| \frac{\partial}{\partial \eta}  \mathcal G(n, k+1)f_{k}(x_2(k, n, \xi, \eta), y(k, n, \eta))\right| 
\\
&\leq \sum_{k\in \Z}| \mathcal G(n, k+1)|\left| \frac{\partial f_{k}(x_2(k, n, \xi, \eta), y(k, n, \eta))}{\partial \eta}\right| \\
&\leq \sum_{k\in \Z}| \mathcal G(n, k+1)|\left(\gamma_k M_{k,n}+\rho_k D_{k,n}\right) <+\infty.\\
\end{split}
\end{displaymath}
In particular, the series
\begin{displaymath}
-\sum_{k\in \Z}  \frac{\partial}{\partial \eta}  \mathcal G(n, k+1)f_{k}(x_2(k, n, \xi, \eta), y(k, n, \eta))
\end{displaymath}
converges uniformly. Thus, it coincides with $\frac{\partial \bar h_n(\xi, \eta) }{\partial \eta}$ and $\eta\to \bar h_n(\xi,\eta)$ is $C^1$. Consequently, $\eta\to \bar H_n(\xi,\eta)$ is $C^1$ as claimed. The proof of the theorem is completed.
\end{proof}

\begin{theorem}\label{theo: reg H second var}
Suppose hypothesis \ref{H: fn is Lipschitz and bounded}, \ref{H: N finite}, \ref{H: q smaller 1}, \ref{H: Aj gamma j smaller 1}, \ref{H: fn diff first var}, \ref{H: fn is Lipschitz second var}, \ref{H: gn is Lipschitz}, \ref{H: sigma rho smaller than 1}, \ref{H: fn diff second var} and \ref{H: gn diff} are satisfied. Moreover, given $n\in \Z$ assume that hypothesis \ref{H: Kn+Jn smaller 1} and \ref{H: temporary name} are satisfied for $n$. Then, the map $\eta \to H_n(\xi,\eta)$ is $C^1$.
\end{theorem}
\begin{proof}
The proof is similar to the proof of Theorem \ref{theo: reg H}. By Theorem \ref{theo: BDP disc}, $H_n(\xi, \eta)=(\xi+h_n(\xi, \eta), \eta)$ we need to prove that the map $\eta\to h_n(\xi, \eta)$ is $C^1$. As in the proof of Theorem~\ref{theo: reg H} we get that $\left|\frac{\partial \bar h_n(\xi, \eta) }{\partial \xi}\right|<1$ for every $(\xi,\eta)\in X\times Y$. Hence,  the operator $\Id +\frac{\partial \bar h_n(\xi, \eta) }{\partial \xi}$ is invertible. Let
\begin{equation}\label{eq: def of R tilde}
\tilde R:=-\left(\Id +\frac{\partial \bar h_n(\xi +h_n(\xi, \eta), \eta) }{\partial u}\right)^{-1}\frac{\partial \bar h_n(\xi +h_n(\xi, \eta),\eta) }{\partial v},
\end{equation}
where $\frac{\partial \bar h_n(\cdot,\cdot) }{\partial u}$ and $\frac{\partial \bar h_n(\cdot,\cdot) }{\partial v}$ denote the derivative of $\bar h_n(\cdot,\cdot)$ with respect to the first and second variables, respectively. 

Now, using that $h_n(\xi,\eta)=-\bar h_n(\xi+h_n(\xi,\eta),\eta)$ (see~\eqref{2131}) together with the differentiability of  $\bar{h}_n$  with respect to the first and second variable (see Theorems \ref{theo: reg H bar} and \ref{theo: reg H bar second var}) and the continuity of $\eta \to h_n(\xi,\eta)$, we may proceed as in the proof of Theorem \ref{theo: reg H} and conclude that $\eta \to h_n(\xi,\eta)$ is differentiable and that $\frac{\partial h_n(\xi, \eta) }{\partial \xi}=\tilde R$. Moreover, it follows from the expression of $\tilde R$ given in \eqref{eq: def of R tilde} and the fact that $\xi \to \bar h_n(\xi,\eta)$ and $\eta \to \bar h_n(\xi,\eta)$ are $C^1$ that $\eta \to h_n(\xi,\eta)$ is also $C^1$ as claimed. 
\end{proof}

\begin{example}\label{end}
Let us discuss an example under which~\ref{H: temporary name} holds. Suppose that 
\[
M:=\sup_{k\in \Z} \max \{ |A_k|, |A_k^{-1}|, |P_k|, \tau_k, \sigma_k \} <+\infty. 
\]
Furthermore, assume that
\[
\rho_k \le \tau_k, \quad \gamma_k \le \frac{1}{(3M)^{2|k|+1}} \quad \text{and} \quad \rho_k \le \frac{1}{2^{|k|} M^{2|k|}},
\]
for $k\in \Z$. By arguing as in Example~\ref{remm}, one can easily show that~\ref{H: temporary name} holds for each $n\in \Z$.
\end{example}

\section{Acknowledgements}
We would like to thank the referees for their useful comments that helped us to improve our paper.
 L.B. was partially supported by a CNPq-Brazil PQ fellowship under Grant No. 306484/2018-8. D. D. was supported in part by Croatian Science Foundation under the project IP-2019-04-1239 and by the University of Rijeka under the projects uniri-prirod-18-9 and uniri-prprirod-19-16.


\begin{thebibliography}{99}
\bibitem{AW} B. Aulbach and T. Wanner, \emph{Topological simplification of nonautonomous difference equations}, J. Difference Equ. Appl. \textbf{12} (2006), 283–296.
\bibitem{BD} L. Backes and D. Dragi\v cevi\' c, \emph{A generalized Grobman-Hartman theorem for nonautonomous
dynamics}, Collect. Math. \textbf{73} (2022), 411--431.
\bibitem{BDP} L. Backes, D. Dragi\v cevi\' c and K. J. Palmer, \emph{Linearization and H\"older Continuity for Nonautonomous Systems},  J. Differential Equations, {\bf 297} (2021), pp 536--574.
\bibitem{Bel73}
G. R. Belitskii,\emph{ Functional equations and the conjugacy of diffeomorphism of finite smoothness class},
 Funct. Anal. Appl. {\bf 7} (1973), 268-277.
\bibitem{Bel78}
G. R. Belitskii,\emph{ Equivalence and normal forms of germs of smooth mappings}, Russian Math. Surveys {\bf 33} (1978),
107-177.
\bibitem{CR} A. Casta\~{n}eda and G. Robledo, \emph{
Differentiability of Palmer's linearization theorem and converse result for density function},
J. Differential Equations {\bf 259} (2015), 4634-4650.
\bibitem{CMR}  A. Casta\~{n}eda, P. Monzon and G. Robledo,
\emph{Nonuniform contractions and density stability results via a smooth topological equivalence},
preprint, https://arxiv.org/abs/1808.07568
\bibitem{CJ}
A. Casta\~{n}eda and N. Jara, \emph{A note on the differentiability of Palmer's topological equivalence for discrete systems}, https://arxiv.org/pdf/2104.14592.pdf
\bibitem{C} W. A. Coppel, \emph{Dichotomies in Stability Theory}, Lect. Notes Math., vol. 629, Springer, Berlin/New York (1978).
\bibitem{CDS} L. V. Cuong, T. S. Doan and S. Siegmund, \emph{A Sternberg theorem for nonautonomous differential
equations},  J. Dynam. Diff. Eq. \textbf{31} (2019), 1279--1299.
\bibitem{DZZ} D. Dragi\v cevi\' c, W. Zhang and W. Zhang, \emph{Smooth linearization of nonautonomous difference equations with a nonuniform dichotomy}, Math. Z. \textbf{292} (2019), 1175--1193.
\bibitem{DZZ2} D. Dragi\v cevi\' c, W. Zhang and W. Zhang, \emph{Smooth linearization of nonautonomous differential equations with a nonuniform dichotomy}, Proc. Lond. Math. Soc. \textbf{121} (2020), 32--50.
\bibitem{E1}
M. S. ElBialy, \emph{Local contractions of Banach spaces and spectral gap conditions},  J. Funct. Anal. {\bf 182} (2001), 108-150.
\bibitem{E2} M. S. ElBialy,\emph{ Smooth conjugacy and linearization near resonant fixed points in Hilbert spaces},  Houston J. Math. \textbf{40} (2014), 467--509.
\bibitem{G1} D. Grobman, \emph{Homeomorphism of systems of differential equations}, Dokl. Akad. Nauk SSSR \textbf{128} (1959) 880--881.
\bibitem{G2} D. Grobman, \emph{Topological classification of neighborhoods of a singularity in $n$-space}, Mat. Sb. (N.S.) \textbf{56}  (1962), 77--94.
\bibitem{H1} P. Hartman, \emph{A lemma in the theory of structural stability of differential equations}, Proc. Amer. Math. Soc. \textbf{11} (1960) 610--620.
\bibitem{H2}  P. Hartman, \emph{On the local linearization of differential equations}, Proc. Amer. Math. Soc. \textbf{14} (1963) 568--573.
\bibitem{Jiang} L. Jiang, \emph{Generalized exponential dichotomy and global linearization}, J. Math. Anal. Appl. \textbf{315} (2006), 474--490.
\bibitem{Jiang2} L. Jiang,  \emph{Ordinary dichotomy and global linearization}, Nonlinear Anal. \textbf{70} (2009), 2722--2730.
\bibitem{Lin} F. Lin, \emph{Hartman's linearization on nonautonomous unbounded system}, Nonlinear Anal. \textbf{66} (2007), 38--50.
\bibitem{Palis} J. Palis, \emph{On the local structure of hyperbolic points in Banach spaces}, An. Acad. Brasil. Cienc. \textbf{40} (1968) 263–266.
\bibitem{Palmer}
K. Palmer, \emph{A generalization of Hartman's linearization theorem}, J. Math. Anal. Appl. {\bf 41} (1973), 753-758.
\bibitem{Pugh}C. Pugh, \emph{On a theorem of P. Hartman}, Amer. J. Math. \textbf{91} (1969) 363–367.
\bibitem{Reinf} A. Reinfelds, \emph{Grobman's-Hartman's theorem for time-dependent difference equations}, Math. Differ. equ. (Russian), 9-13, Latv. Univ. Zinat. Raksti, \textbf{605}, Latv. Univ., Riga, 1997.
\bibitem{RS} A.A. Reinfelds and D. Steinberga, \emph{Dynamical equivalence of quasilinear equations}, Internat. J. Pure Appl. Math. \textbf{98} (2015), 355--364.
\bibitem{R-S-JDDE04} H. M. Rodrigues and J. Sol${\rm \grave{a}}$-Morales, \emph{Smooth linearization for a saddle on Banach spaces},
 J. Dyn. Differential Equations {\bf 16} (2004), 767-793.
\bibitem{R-S-JDDE06} H. M. Rodrigues and J. Sol${\rm \grave{a}}$-Morales, \emph{Invertible contractions and asymptotically stable
                    ODE's that are not $C^1$-linearizable},  J. Dyn. Differential Equations {\bf 18} (2006), 961-974.
\bibitem{SX} J.L. Shi and K. Q. Xiong, \emph{On Hartman's Linearization Theorem and Palmer's Linearization Theorem}, J. Math. Anal. Appl. \textbf{192} (1995), 813--832.
\bibitem{Stern}
S. Sternberg, \emph{Local contractions and a theorem of Poincar\'e},   Amer. J. Math. {\bf 79} (1957), 809-824.
\bibitem{Stern2}
S. Sternberg, \emph{On the structure of local homeomorphisms of Euclidean $n$-space},  Amer. J. Matth. \textbf{80} (1958), 623-631.
\bibitem{Stri-JDE90} S. van Strien, \emph{Smooth linearization of hyperbolic fixed points without resonance conditions},
                    J. Differential Equations \textbf{85} (1990), 66-90.
\bibitem{XWKR} Y. H. Xia, R. Wang, K. I. Kou and D. O'Regan, \emph{On the linearization theorem for nonautonomous differential equations}, Bull. Sci. Math. \textbf{139} (2015), 829--846.
\bibitem{ZLZ} W. Zhang, K. Lu and W. Zhang, \emph{Differentiability of the conjugacy in the Hartman-Grobman theorem}, Trans. Amer. Math. Soc. \textbf{369} (2017), 4995–5030.
\bibitem{ZhangZhang14JDE}
W. M. Zhang and W. N. Zhang, \emph{Sharpness for $C^1$ linearization of planar hyperbolic diffeomorphisms},
J. Differential Equations {\bf 257} (2014),
4470-4502.
\bibitem{ZZJ} W. M. Zhang, W. N. Zhang and W. Jarczyk, \emph{Sharp regularity of linearization for $C^{1, 1}$
hyperbolic diffeomorphisms},  Math. Ann. {\bf 358} (2014), 69-113.


\end{thebibliography}
\end{document}